\documentclass{tac}

\usepackage{xypic}
\usepackage{amsmath,amssymb}

\author {Francis Borceux, Stefaan Caenepeel and George Janelidze}

\thanks{The second author is partially supported by research project G.0622.06 ``Deformation quantization methods
for algebras and categories with applications to quantum mechanics" from
FWO-Vlaanderen, and the third author is partially supported by South African NRF.}

\address{D\'epartement de Math\'ematique, Universit\'e Catholique de Louvain, 
Chemin du Cyclotron 2, B-1347 Louvain-la-Neuve, Belgium\\[5pt]
Faculty of Engineering,
Vrije Universiteit Brussel, Pleinlaan 2, B-1050 Brussels, Belgium\\[5pt]
 Department of Mathematics and Applied Mathematics, University of Cape Town, Rondebosch 7701, South Africa}

\title {Monadic approach to Galois descent and cohomology}
\dedication{Dedicated to Dominique Bourn at the occasion
of his sixtieth birthday.}

\copyrightyear{2008}


\eaddress{borceux@math.ucl.ac.be\CR scaenepe@vub.ac.be\CR janelg@telkomsa.net}

\newtheorem{proposition}{Proposition}[section]
\newtheorem{lemma}[proposition]{Lemma}
\newtheorem{corollary}[proposition]{Corollary}
\newtheorem{theorem}[proposition]{Theorem}

\theoremstyle{definition}
\newtheorem{definition}[proposition]{Definition}
\newtheorem{example}[proposition]{Example}

\newtheorem{observation}[proposition]{Observation}
\newtheorem{situation}[proposition]{Situation}

\theoremstyle{remark}
\newtheorem{remark}[proposition]{Remark}

\newcommand{\thlabel}[1]{\label{th:#1}}
\newcommand{\thref}[1]{Theorem~\ref{th:#1}}
\newcommand{\selabel}[1]{\label{se:#1}}
\newcommand{\seref}[1]{Section~\ref{se:#1}}
\newcommand{\lelabel}[1]{\label{le:#1}}
\newcommand{\leref}[1]{Lemma~\ref{le:#1}}
\newcommand{\prlabel}[1]{\label{pr:#1}}
\newcommand{\prref}[1]{Proposition~\ref{pr:#1}}
\newcommand{\colabel}[1]{\label{co:#1}}
\newcommand{\coref}[1]{Corollary~\ref{co:#1}}
\newcommand{\relabel}[1]{\label{re:#1}}
\newcommand{\reref}[1]{Remark~\ref{re:#1}}
\newcommand{\exlabel}[1]{\label{ex:#1}}

\newcommand{\delabel}[1]{\label{de:#1}}
\newcommand{\deref}[1]{Definition~\ref{de:#1}}
\newcommand{\eqlabel}[1]{\label{eq:#1}}
\newcommand{\equref}[1]{(\ref{eq:#1})}
\newcommand{\silabel}[1]{\label{si:#1}}
\newcommand{\siref}[1]{Situation~\ref{si:#1}}

\def\mapright#1{\smash{\mathop{\longrightarrow}\limits^{#1}}}

\def\equal#1{\smash{\mathop{=}\limits^{#1}}}

\newcommand{\Fam}{{\rm Fam}}

\newcommand{\Pic}{{\rm Pic}}

\newcommand{\End}{{\rm End}}

\newcommand{\Aut}{{\rm Aut}}

\def\ot{\otimes}

\def\aA{{\rm \bf A}}
\def\xX{{\rm \bf X}}
\def\yY{{\rm \bf Y}}
\def\uU{{\rm \bf U}}
\def\cC{{\rm \bf C}}
\def\dD{{\rm \bf D}}
\def\eE{{\rm \bf E}}
\def\sS{{\rm \bf S}}
\def\mM{{\rm \bf M}}

\def\Cat{{ \rm\bf Cat}}
\def\Mod{{ \rm\bf Mod}}
\def\GL{{ \rm\bf GL}}
\def\PGL{{ \rm\bf PGL}}
\def\Alg{{ \rm\bf Alg}}
\newcommand{\Sets}{{\rm \bf Sets}}

\def\ZZ{{\mathbb Z}}

\def\*C{{}^*\hspace*{-1pt}{\cC}}

\def\text#1{{\rm {\rm #1}}}

\def\ul{\underline}

\begin{document}

\maketitle
\begin{abstract}
 We describe a simplified categorical approach to Galois descent theory. It is well known that Galois descent is a special case of Grothendieck descent, and that under mild additional conditions the category of Grothendieck descent data coincides with the Eilenberg-Moore category of algebras over a suitable monad. This also suggests using monads directly, and our monadic approach to Galois descent makes no reference to Grothendieck descent theory at all. In order to make Galois descent constructions perfectly clear, we also describe their connections with some other related constructions of categorical algebra, and make various explicit calculations, especially with 1-cocycles and 1-dimensional non-abelian cohomology, usually omitted in the literature.
\end{abstract}


\section*{Introduction}\selabel{0}
The so-called Galois descent theory is an old tool used in algebraic geometry, algebraic number theory, and related topics of ring and module theory. It allows to describe certain structures of algebra and geometry defined over a base field B using similar structures defined over a Galois extension E of B. Its results are {\it descent theorems} that are of the form\\
\hspace*{1cm}{\it The category of $B$-structures is canonically equivalent}\\
    \hspace*{1cm}{\it       to the category of descent data of $E$-structures}\\
usually supplied by descriptions of the set of isomorphism classes of $B$-structures with a fixed extension to $E$ in terms of 1-dimensional cohomology of the Galois group 
of the Galois extension $E/B$. If one replaces the Galois extension $E/B$ with an arbitrary morphism 
$p :\ E \to B$ in a category $\cC$ equipped with a Grothendieck fibration $(\aA,F)$ of categories over $\cC$, there is still a {\it canonical comparison functor}
$$K^p:\ \aA^B\to {\rm Des}(p)$$
from the category ${\rm Des}(p)$ of suitably defined {\it descent data} to the fibre 
$\aA^B$ of $(\aA,F)$ over $B$. In the Galois descent constructions the role of $\cC$ is usually played by the dual category of commutative rings, which explains the direction of the arrow $p$ (and this is why we use the opposite direction in our \seref{4}). According to Grothendieck descent theory, the morphism $p :\ E \to B$ is said to be an {\it effective descent morphism} if $K^p$ is a category equivalence. Under certain additional conditions on $(\aA,F)$ that hold in many known examples, $p :\ E \to B$ is an effective descent morphism if and only if the change-of-base functor
$$p^*:\ \aA^B\to \aA^E$$
is monadic. This follows from the fact that $K^p$ becomes nothing but the comparison functor, in the sense of monad theory, from $\aA^B$ to the category of algebras over the monad determined by $p^*$ and its left adjoint.

In this paper we develop a simplified categorical approach to Galois descent by presenting the category of Galois descent data as the category of algebras over a suitable monad directly, i.e. not involving the intermediate context of Grothendieck fibrations and Grothendieck descent. We also show that the monadic approach is helpful in presenting the ($1$-dimensional) cohomological side of the story by transforming descent data into suitable $1$-cocycles - again directly, i.e. not using Grothendieck's cocycle condition. In spite of our goal of simplification, we present some intermediate categorical constructions and various explicit calculations usually omitted in the literature; we believe they are necessaryÉ. On the other hand we omit some other important constructions and their natural generalizations; e.g. we do not consider torsors, do not discuss generalizations involving categorical Galois theory whose Galois groupoids are internal groupoids in an abstract category, and do not even consider infinite Galois extensions of fields.

\seref{1} begins with $1$-dimensional cocycles and cohomology  $H^1(G,A)$ of a group 
$G$ with abelian and then with non-abelian coefficients. We then extend it replacing a $G$-group $A$ with a $G$-category $\xX$, but also show that in some case the $G$-category case reduces to the 
$G$-group case (\prref{1.4}). Next, we show that involving $2$-dimensional category theory makes all constructions more natural. In particular the category $Z^1(G,\xX)$ of $1$-cocycles 
$G \to \xX$ becomes nothing but the category ${\rm PsNat}({\bf 1},\xX)$ of pseudo-natural transformations ${\bf 1}\to \xX$, and the reduction above can be obtained from a canonical factorization of pseudo-natural transformations ${\bf 1}\to \xX$.

\seref{2} presents a $G$-category structure on $\xX$ as a special (``split") monoid in the monoidal category ${\rm Fam}(\xX^\xX)$ of families of endofunctors of $\xX$. 
A $1$-cocycle $G\to \xX$ then becomes an action of $\xX$ on an object in $\xX$. 
The coproduct functor ${\rm Fam}(\cC) \to \cC$, in the case of $\cC$
 being the monoidal category $\xX^\xX$, transforms the monoid above into a monad on $\xX$, provided $\xX$ admits $G$-indexed coproducts. That monad, denoted by 
 $G(-)$, plays a crucial role in Galois descent theory. It is in fact extensively used in Galois theory of commutative rings, where, for a Galois extension $E/B$, $GE$ is the algebra of all maps from $G$ to $E$; however, as far as we know, it was never considered there as a monad. Note also, that involving the monoidal category ${\rm Fam}(\xX^\xX)$ (instead of just 
 $\xX^\xX$) is not absolutely necessary, but makes all constructions more elegant since the monoidal structure of ${\rm Fam}(\xX^\xX)$ is automatically distributive with respect to coproducts, while in $\xX^\xX$ it only happens for coproducts of very special objects. 
 \seref{2} ends with \thref{2.5}, according to which a $1$-cocycle $G\to \xX$ is nothing but a $G(-)$-algebra.
 
\seref{3} begins by comparing split and connected monoids in ${\rm Fam}(\cC)$,
where ``split" (as above) means being ``a family of unit objects" while ``connected" means 
being ``inside $\cC$", i.e. being ``a one-member family". This allows to construct a canonical morphism from $G(-)$ to any monad $T$ on $\xX$ that is invariant (in a suitable sense) under the 
$G$-action, and $T$ is said to be a $G$-Galois monad if that canonical morphism is an isomorphism. A monadic adjunction $(F,U,\eta,\varepsilon)$ from $\xX$ to $\yY$ is then called a $G$-Galois adjunction if it determines a $G$-Galois monad on $\xX$. The main descent theorem (\thref{3.8}) that establishes and describes the equivalence $\yY\sim Z^1(G,\xX)$ for a Galois adjunction, is a straightforward consequence of monadicity in this context. In particular $Z^1(G,\xX)$ plays the role of the category of descent data. The second main result (\thref{3.9}) describes, using \thref{3.8} and abovementioned \prref{1.4}, the full subcategory $\yY[Y_0]$ of $\yY$ with objects all 
$Y$ in $\yY$ with $U(Y) \cong U(Y_0)$ for a fixed $Y_0$, as the category of $1$-cocycles 
$G \to\End(U(Y_0))$, where the monoid $\End(U(Y_0))$ is equipped with a $G$-action induced (in a suitable sense) by the $1$-cocycle $\xi_0$ corresponding to $Y_0$. Accordingly, \thref{3.9} also describes the set of isomorphic classes of objects in $\yY[Y_0]$, presented its bijection with 
$H^1(G,\End(U(Y_0))_{\xi_0}) = H^1(G,\Aut(U(Y_0))_{\xi_0})$.

\seref{4} applies our results to the adjunction 
$(E\hbox{-}{\bf \rm Mod})^{\rm op}\to (B\hbox{-}{\bf \rm Mod})^{\rm op}$
determined by a homomorphism $p:\ B\to E$ of commutative rings, where $E$ is equipped with a right $G$-action making $(E\hbox{-}{\bf \rm Mod})^{\rm op}$ a $G$-category. We show that this adjunction is a $G$-Galois adjunction if and only if $p :\ B \to E$ is a $G$-Galois extension of commutative rings in the sense of \cite{CHR}; in particular this implies that $G$ must be finite. After that we present the Galois descent/cohomology results obtained as direct translations of Theorems \ref{th:3.8}
and \ref{th:3.9}. The rest of \seref{4} is devoted to three classical examples. Much more examples could be presented of course, but we only wanted to explain how the general theory works in practice.

Finally, let us mention that it is not really clear how old is the idea of Galois descent. Apart from Grothendieck's work, relatively recent sources (among many others) are \cite{S2} (mentioning ``descent for forms"), \cite{KO}, \cite{Gre}, \cite{Jah}. As it was pointed out to the third named author during his talk in Budapest by Tam\'as Szamuely, Galois descent is also mentioned by J.-P. Serre in \cite{S1} with a reference to \cite{W}, although Serre himself also mentions there an older work of F. Ch\^atelet.

\section{$G$-categories and cohomology}\selabel{1}
\setcounter{equation}{0}
\subsection*{Ordinary non-abelian cohomology of groups}\hfill\break
Let $G$ be a group, and $A$ a $G$-module - which is the same as a module over the
integral group ring $\ZZ[G]$. Classically, a map $\varphi:\ G\to A$ is said to be a 1-{\it cocycle}
if
\begin{equation}\eqlabel{1.1}
\varphi(hg)=\varphi(h)+h\varphi(g)
\end{equation}
for all $g,h\in G$. The set $Z^1(G,A)$ of all $1$-cocycles forms an abelian group under
the pointwise addition, and the map $A\to Z^1(G,A)$ defined by
$a\mapsto (g\mapsto ga-a)$ is a group homomorphism. 
Denoting the image of that map by $B^1(G,A)$ (=the group of {\it boundaries}) one then defines the 
$1$-dimensional cohomology group $H^1(G,A)$ as the quotient group
\begin{equation}\eqlabel{1.2}
H^1(G,A)=Z^1(G,A)/B^1(G,A).
\end{equation}
More generally, this extends to the {\it non-abelian} context as follows:\\
$\bullet$ First of all the non-abelian version of the notion of $G$-module is to be chosen. Since a $G$-module is the same as a $G$-object in the category of abelian groups, the obvious candidate is a $G$-group, i.e. a $G$-object in the category of groups. Such a $G$-group $A$ is a group $A$ equipped with a $G$-action satisfying
\begin{equation}\eqlabel{1.3}
g(ab)=g(a)g(b)
\end{equation}
for all $g\in G$ and $a,b\in A$. Equivalently, a $G$-group is an internal group (=group object) in the category of $G$-sets.\\
$\bullet$ We then repeat the definition of a $1$-cocycle $\varphi:\ G\to A$ with
\begin{equation}\eqlabel{1.4}
\varphi(hg)=(\varphi(h))(h\varphi(g))
\end{equation}
instead of \equref{1.1}. However, now the set $Z^1(G,A)$ does not have any natural group
structure - it is just a pointed set with the trivial group homomorphism $G\to A$ being the distinguished point.\\
$\bullet$ Yet, the group $A$ acts on $Z^1(G,A)$ via
\begin{equation}\eqlabel{1.5}
(a\varphi)(g)=a(\varphi(g))(ga)^{-1},
\end{equation}
and we can now define the non-abelian $H^1(G,A)$ as the pointed set
\begin{equation}\eqlabel{1.6}
H^1(G,A)=Z^1(G,A)/A
\end{equation}
of orbits.

Another well-known way to do this is to regard $Z^1(G,A)$ as a groupoid, and then define 
$H^1(G,A)$ as the set
\begin{equation}\eqlabel{1.7}
H^1(G,A)=\pi_0(Z^1(G,A))
\end{equation}
of its connected components. Accordingly the objects of $Z^1(G,A)$ are $1$-cocycles
$G\to A$, and a morphism $\varphi\to \psi$ is an element $a$ in $A$ such that
$a\varphi=\psi$ under the action \equref{1.5}.

\subsection*{From $G$-groups to $G$-categories}\hfill\break
We are now going to replace a $G$-group $A$ by a $G$-category X defined as a $G$-object in the category $\Cat$ of all categories. We will hence consider X as an ordinary category equipped with a 
$G$-indexed family of functors 
\begin{equation}\eqlabel{1.8}
g(-):\ \xX\to \xX
\end{equation}
with $1(-)=1_\xX$ and $(hg)(-)=h(-)g(-)$. We introduce

\begin{definition}\delabel{1.1}
Let $G$ be a group and $\xX$ a $G$-category. Then:
\begin{enumerate}
\item[(a)] A $1$-cocycle $(X,\xi):\ G\to \xX$ is an object $X$ in $\xX$ together with a family
$\xi=(\xi_g)_{g\in G}$ of morphisms $\xi_g:\ g(X)\to X$ in $\xX$ satisfying
$\xi_1=1_X$ and $\xi_{hg}=\xi_h h(\xi_g)$ for all $g,h\in G$.
\item[(b)] A morphism $f:\ (X,\xi)\to (X',\xi')$ of $1$-cocycles is a morphism $f:\ X\to X'$
in $\xX$ such that $f\xi_g=\xi'_g g(f)$, for all $g\in G$; the thus obtained category
of $1$-cocycles $G\to \xX$ will be denoted by $Z^1(G,\xX)$.
\item[(c)] The isomorphism class of an object $(X,\xi)$ in $Z^1(G,\xX)$ will be denoted by 
$[X,\xi]$, and the set of all such classes by $H^1(G,\xX)$; we will call it the $1$-dimensional cohomology set of the group $G$ with coefficients in the $G$-category.
\end{enumerate}
\end{definition}

This indeed generalizes the definition for groups, but we have to make

\begin{remark}\relabel{1.2}
(a) Let $(X,\xi):\ G\to \xX$ be a $1$-cocycle. Then each $\xi_g$ is an isomorphism.
Indeed, applying $\xi_1=1_X$ and $\xi_{hg}=\xi_h h(\xi_g)$ to the case $h=g^{-1}$,
we conclude that each $\xi_{g^{-1}}$ is a split epimorphism and each
$g^{-1}(\xi_g)$ is a split monomorphism; since every element of $G$ can be presented as a 
$g^{-1}$, and all functors of the form $g(-)$ ($g \in G$) are isomorphisms, this tells us that each 
$\xi_g$ is a split epimorphism and a split monomorphism at the same time, i.e. each $\xi_g$ is an isomorphism. This yields an alternative definition of a $1$-cocycle:
if we require each $\xi_g$ to be an isomorphism (or just $\xi_1$ to be an isomorphism), 
then $\xi_1=1_X$ can be deduced from $\xi_{hg}=\xi_h h(\xi_g)$.
Indeed, applying it to $g = h = 1$ we obtain $\xi_1\xi_1=\xi_1$, and so $\xi_1$ is an identity whenever it is an isomorphism. In particular, when $\xX$ is a groupoid (i.e. every morphism in it is an isomorphism) we can simply omit the condition $\xi_1=1_X$ in \deref{1.1}(a).

(b) If $\xX$ is a groupoid, then obviously $Z^1(G,\xX)$ also is a groupoid, and in this case 
$H^1(G,\xX)$ is the same as $\pi_0(Z^1(G,\xX))$, similarly to \equref{1.7}.

(c) When $\xX$ has only one object, a $1$-cocycle $(X,\xi):\ G\to \xX$ becomes just a map
$\xi$ from $G$ to the monoid of endomorphisms of the unique object of $\xX$
written as $g\mapsto \xi_g$ satisfying $\xi_1=1$ and $\xi_{hg}=\xi_h h(\xi_g)$ 
above. If in addition that monoid was a group, then (by (a)) we can omit the first of these two equalities, and the second one becomes precisely \equref{1.4} if we write $\varphi(g)$ for $\xi_g$, etc.
Thus, when $\xX$ is a $G$-group $A$ regarded as a one object 
$G$-category, the category $Z^1(G,\xX)$ and the set $H^1(G,\xX)$ coincide with the ordinary (groupoid) $Z^1(G,A)$ and (set) $H^1(G,A)$ respectively.

(d) Unlike the case of a $G$-group, the general $H^1(G,\xX)$ has no distinguished point; in order to get such a point we have to take $\xX$ itself to be pointed, i.e. to have a distinguished object $X_0$ such that $g(X_0) = X_0$ for all $g\in G$.
\end{remark}

\subsection*{From $G$-categories to $G$-groups}\hfill\break
Given an object $X$ in a $G$-category $\xX$, let us compare:\\
$\bullet$ the $1$-cocycles $(X,\xi):\ G\to \xX$ with the same fixed $X$; their full subcategory
in $Z^1(G,\xX)$ will be denoted by $Z^1(G,\xX)_X$, and the corresponding subset in 
$H^1(G,\xX)$ will be denoted by $H^1(G,\xX)_X$;\\
$\bullet$ the $1$-cocycles $G\to \End(X)_\rho$, where $\End(X)_\rho$
is the endomorphism monoid of $X$ considered as a category (= the full subcategory in $X$ with the unique object $X$) and equipped with a $G$-category structure $\rho$;\\
$\bullet$ the $1$-cocycles $G\to \Aut(X)_\rho$, where $\Aut(X)_\rho$ is the automorphism group 
$\Aut(X)$ of $X$ in $\xX$ equipped with a $G$-group structure $\rho$.

\begin{observation}\label{obs1.3}
Clearly the $G$-category structure of $\xX$ restricts to a $G$-category structure on 
$\End(X)$ if and only if the object $X$ is $G$-invariant. However there is another way to make such a structure. For any $1$-cocycle $(X,\xi):\ G\to \xX$ (with the same $X$), $g\in G$, and an endomorphism
$a:\ X\to X$, we define $ga=\xi_g g(a)(\xi_g)^{-1}$ (writing $ga$ for the new 
$G$-action in distinction from the old $g(a)$). We will denote this structure by $\rho(\xi)$,
and write $Z^1(G,\End(X)_\xi)$ for the corresponding category of $1$-cocycles and
$H^1(G,\End(X)_\xi)$ for the corresponding $1$-dimensional cohomology set respectively. The action 
$\rho(\xi)$ clearly restricts to a $G$-group structure on $\Aut(X)$; we will then write
$Z^1(G,\Aut(X)_\xi)$ and $H^1(G,\Aut(X)_\xi)$ respectively. Of course $Z^1(G,\Aut(X)_\xi)$
is just the groupoid whose morphisms are all isomorphisms in 
$Z^1(G,\End(X)_\xi)$, and therefore $H^1(G,\Aut(X)_\xi)$ is the same as $H^1(G,\End(X)_\xi)$
\end{observation}

After this a straightforward calculation proves

\begin{proposition}\prlabel{1.4}
Let $(X,\xi):\ G\to \xX$ be a fixed $1$-cocycle. The assignment
$\varphi\mapsto \bigl(X, (\varphi(g)\xi_g)_{g\in G}\bigr)$ extends to a category isomorphism
$Z^1(G,\End(X)_\xi)\cong Z^1(G,\xX)_X$, and in particular
$$H^1(G,\Aut(X)_\xi)= H^1(G,\End(X)_\xi)\cong H^1(G,\xX)_X.$$
\end{proposition}

\begin{remark}\relabel{1.5}
\prref{1.4} {\it almost} reduces the description of $H^1(G,\xX)$ 
to the case where $\xX$ is a $G$-group. Since $H^1(G,\xX)$ 
obviously is the disjoint union of $H^1(G,\xX)_X$ for all objects $X$ from any fixed skeleton of the category $\xX$, and each $H^1(G,\xX)_X$ is either empty, or bijective to
$H^1(G,\Aut(X)_\xi)$ for some $\xi$, it gives a bijection of the form
\begin{equation}\eqlabel{1.9}
H^1(G,\xX)\cong \sqcup_{(X,\xi)\in \Xi} H^1(G,\Aut(X)_\xi),
\end{equation}
where $\Xi$ is any set of $1$-cocycles $(X,\xi):\ G\to \xX$ 
satisfying the following conditions:

(a) if $(X,\xi)$ and $(X',\xi')$ are in $\Xi$, then $X$ and $X'$ are not isomorphic in
$\xX$;

(b) $\Xi$ is a maximal set with the property (a).
\end{remark}

\subsection*{Involving 2-dimensional category theory}\hfill\break
Consider the following two questions:\\
{\bf Question A}: What is an appropriate notion of a morphism of $G$-categories?\\
{\bf Question B}: Are the ordinary and the generalized notions of 1-cocycle merely technical, or they occur as special cases of natural categorical constructions?

Considering $G$-categories as $G$-objects in $\Cat$, or as internal categories in the category 
$\Sets^G$ of $G$-sets, one would of course define the morphisms between them as the 
G-functors, i.e. as the morphisms of $G$-objects in $\Cat$, or, equivalently, as internal functors in 
$\Sets^G$. However, thinking of a $G$-category as a functor from $G$ to $\Cat$, and having in mind that $\Cat$ is a 2-dimensional category, we have a choice between the natural transformations Ð which will give us the $G$-functors again, and the so-called {\it pseudo-natural transformations}. The second choice suggests the following answer to Question A:

\begin{definition}\delabel{1.6}
Let $\cC$ and $\dD$ be $G$-categories. A morphism from $\cC$ to $\dD$ is a pair 
$(F,\xi)$, in which $F$ is a functor from $\cC$ to $\dD$ considered as ordinary categories, and
$\xi=(\xi_{g,C})_{g\in G,C\in {\rm Ob}(\cC)}$ a family of isomorphisms
$\xi_{g,C}:\ g(F(C))\to F(g(C))$ natural in $C$ and satisfying
$\xi_{1,C}={\bf 1}_{F(C)}$ and $\xi_{hg,C}=\xi_{h,g(C)}h(\xi_{g,C})$ for all objects $C\in \cC$
and all $g,h\in G$.
\end{definition}

Furthermore, there are morphisms between pseudo-natural transformations called modifications, and for given $\cC$ and $\dD$, they together form a category ${\rm PsNat}(\cC,\dD)$. The objects of 
${\rm PsNat}(\cC,\dD)$ are thus the morphisms from $\cC$ to $\dD$ in the sense of \deref{1.6}, and a morphism $f:\ (F,\xi)\to (F',\xi')$ in ${\rm PsNat}(\cC,\dD)$ is a natural transformation
$f:\ F\to F'$ such that
\begin{equation}\eqlabel{1.10}
f_{g(C)}\xi_{g,C} =\xi'_{g,C}g(f_C).
\end{equation}
In particular if we take $\cC={\bf 1}$ (the category with only one object and one morphism) and 
write $\xX$ instead of $\dD$, then ${\rm PsNat}(\cC,\dD)$ becomes nothing but what we called 
$Z^1(G,\xX)$, yielding a good affirmative answer to our Question B.

\subsection*{2-Functoriality of $Z^1(G,\xX)$ in $\xX$}\hfill\break
In order to make $\xX\mapsto Z^1(G,\xX)$
a 2-functor we need to make one more step into
2-dimensional category theory:\\
Given three $G$-categories $\cC,\dD,\eE$, we construct the composition functor
$$\ot=\ot_{\cC,\dD,\eE}:\ {\rm PsNat}(\dD,\eE)\times {\rm PsNat}(\cC,\dD)\to
{\rm PsNat}(\cC,\eE)$$
as follows\\
$\bullet$ $(\ul{F},\ul{\xi})\ot (F,\xi)=(\ul{F}F,\xi\ot \ul{\xi})$ on objects, where $\ot$ is the
{\it argumentwise pasting operation}, i.e. $\xi\ot \ul{\xi}$ is defined by
$(\xi\ot \ul{\xi})_{g,C}= (\ul{F}(\xi_{g,C}))\ul{\xi}_{g,F(C)}$;\\
$\bullet$ on morphisms the functor $\ot$ is the ordinary {\it horizontal composition} of natural transformations, i.e. for $\ul{f}:\ (\ul{F},\ul{\xi})\to (\ul{F}',\ul{\xi}')$ in ${\rm PsNat}(\dD,\eE)$
and $f:\ (F,\xi)\to (F',\xi')$ in ${\rm PsNat}(\cC,\dD)$, the composite
$\ul{f}\ot f:\ (\ul{F},\ul{\xi})\ot (F,\xi)\to (\ul{F}',\ul{\xi}')\ot (F',\xi')$ is defined by
$(\ul{f}\ot f)_C= (\ul{f}_{F'(C)})(\ul{F}(f_C))$ (which is the same as $(\ul{F}'(f_C))(\ul{f}_{F(C)})$).

The collection of all $G$-categories, their pseudo-natural transformations, and modifications with the composition $\ot$ and the obvious {\it identity pseudo-natural transformations}
$1_{\cC}: \cC\to \cC$ (for all $\cC$) form a 2-dimensional category $G\hbox{-}\Cat$,
in which the hom-categories ${\rm hom}_{G\hbox{-}\Cat}(\cC,\dD)$ are the categories
${\rm PsNat}(\cC,\dD)$. In particular, in the notation above, we have
\begin{equation}\eqlabel{1.11}
Z^1(G,\xX)= {\rm hom}_{G\hbox{-}\Cat}({\bf 1},\xX)={\rm PsNat}({\bf 1},\xX),
\end{equation}
yielding a 2-functor
\begin{equation}\eqlabel{1.12}
Z^1(G,-):\ G\hbox{-}\Cat\to \Cat.
\end{equation}

\subsection*{A remark on $Z^0(G,\xX)$ and $H^0(G,\xX)$}\hfill\break
Let us repeat our constructions ``decreasing the dimension by one". That is, we take the
$G$-category $\xX$ to be a {\it discrete} category, i.e. a category with no nonidentity morphisms,
$\xX={\rm Dis}(S)$ - denoting the $G$-set of its objects by $S$. 
Then it is easy to see that the 
pseudo-natural transformations above become just the natural transformations, the modifications become identities, i.e. $Z^1(G,{\rm Dis}(S))$ is discrete, and we have
\begin{equation}\eqlabel{1.13}
Z^1(G,{\rm Dis}(S))\cong {\rm Dis}\{s\in S~|~gs=s~~{\rm for~all}~~g\in G\}.
\end{equation}
Since the classical definitions of $Z^0(G,S)$ and $H^0(G,S)$ say
\begin{equation}\eqlabel{1.14}
H^0(G,S)=Z^0(G,S)= {\rm Dis}\{s\in S~|~gs=s~~{\rm for~all}~~g\in G\},
\end{equation}
we conclude that $Z^1(G,\xX)$ and $H^1(G,\xX)$ defined for an arbitrary $G$-category $\xX$, generalize not only $Z^1(G,A)$ and $H^1(G,A)$ for a $G$-group $A$, but also 
$Z^0(G,S)$ and $H^0(G,S)$ for a $G$-set $S$.

\subsection*{A categorical construction behind \prref{1.4}}\hfill\break
Any functor $F:\ \cC\to \dD$ canonically factorizes as
\begin{equation}\eqlabel{1.15}
\xymatrix{
\cC\ar[r]^(.4){F^{(1)}}&{\rm\bf fact}(F)\ar[r]^(.65){F^{(2)}}&\dD}
\end{equation}
where:\\
$\bullet$ the objects in ${\rm\bf fact}(F)$ are as $\cC$ and the morphisms as $\dD$, that is\\
${\rm hom}_{{\rm\bf fact}(F)}(C,C')={\rm hom}_{\dD}(F(C),F(C'))$;\\
$\bullet$  the functor $F^{(1)}$ is defined by $F^{(1)}(C) = C$ on objects and 
$F^{(1)}(u) = F(u)$ on morphisms;\\
$\bullet$  the functor $F^{(2)}$ is defined by $F^{(2)}(C) = F(C)$ on objects and 
$F^{(2)}(u) = u$ on morphisms.\\

If $\cC$ and $\dD$ were $G$-categories and $F$ a $G$-functor, then, since the factorization 
\equref{1.15} is functorial, it would make ${\rm\bf fact}(F)$ a $G$-category and $F^{(1)}$ and 
$F^{(2)}$ $G$-functors. More generally, if $(F,\xi):\ \cC\to \dD$ is a morphism of
$G$-categories in the sense of \deref{1.6}, then ${\rm\bf fact}(F)$ can still be regarded
as a $G$-category and there are canonical $\xi^{(1)}$ and $\xi^{(2)}$ making
$(F^{(1)},\xi^{(1)})$ and $(F^{(2)},\xi^{(2)})$ morphisms of $G$-categories. Explicitly:\\
$\bullet$ For each $g\in G$, the functor $g(-):\ {\rm\bf fact}(F)\to {\rm\bf fact}(F)$ is defined as 
$g(-):\ C\to C$ on objects and via
\begin{equation}\eqlabel{1.16}
\xymatrix{
gF(C)\ar[rr]^{\xi_{g,C}}\ar[d]_{g(a)}&&Fg(C)\ar[d]^{g(a)}\\
gF(C')\ar[rr]^{\xi_{g,C'}}&&Fg(C')}
\end{equation}
on morphisms; here the first vertical arrow represents the image under 
$g(-):\ D\to D$ of an arbitrary $a :\ F(C) \to F(C')$ in $\dD$, but the second one is the new $g(a)$ to be defined as a morphism from $Fg(C)$ to $Fg(C')$ in $\dD$.\\
$\bullet$ Since $\xi_{g,C}$ is natural in $C$, it is easy to check that the diagram
\begin{equation}\eqlabel{1.17}
\xymatrix{
\cC\ar[rr]^{F^{(1)}}\ar[d]_{g(-)}&& {\rm\bf fact}(F)\ar[d]^{g(-)}\\
\cC\ar[rr]^{F^{(1)}}&& {\rm\bf fact}(F)}
\end{equation}
commutes for every $g\in G$, and so $\xi^{(1)}$ is to be defined as the appropriate family of identity morphisms.\\
$\bullet$ For a morphism $a:\ C\to C'$ in  {\rm\bf fact}(F), the diagram \equref{1.16} can now be rewritten as
\begin{equation}\eqlabel{1.18}
\xymatrix{
gF^{(2)}(C)\ar[rr]^{\xi_{g,C}}\ar[d]_{gF^{(2)}(a)}&& F^{(2)}g(C)\ar[d]^{F^{(2)}g(a)}\\
gF^{(2)}(C')\ar[rr]^{\xi_{g,C'}}&& F^{(2)}g(C')}
\end{equation}
and accordingly $\xi^{(2)}$ is defined as the family of isomorphisms
$$\xi_{g,C}:\ gF^{(2)}(C)=gF(C)\to Fg(C)= F^{(2)}g(C).$$

\begin{definition}\delabel{1.7}
Let $(F,\xi):\ \cC\to \dD$ be a morphism of $G$-categories. With notation as above, we
will say that $(F,\xi)= (F^{(2)},\xi^{(2)})\ot (F^{(1)},\xi^{(1)})$ is the canonical factorization of
$(F,\xi)$.
\end{definition}

\begin{example}\exlabel{1.8}
Let $(X,\xi):\ G\to \xX$ be a $1$-cocycle as in \prref{1.4} and 
$(F,\xi)= (F^{(2)},\xi^{(2)})\ot (F^{(1)},\xi^{(1)})$ the canonical factorization of the 
corresponding morphism $(F,\xi):\ {\bf 1} \to \xX$. Then\\
(a) the category $ {\rm\bf fact}(F)$ is the same as the monoid $\End(X)$;\\
(b) the $G$-category structure $\rho(\xi)$ on $\End(X)$ described in 
Observation \ref{obs1.3} is the same as the $G$-category structure on $ {\rm\bf fact}(F)$ described now (\equref{1.16});\\
(c) the composite of the isomorphism
$Z^1(G,\End(X)_\xi)\cong Z^1(G,\xX)_X$ described in \prref{1.4} and the embedding
$Z^1(G,\xX)_X\to Z^1(G,\xX)$ is the same as the functor
$$Z^1(G, (F^{(2)},\xi^{(2)})):\ Z^1(G, {\rm\bf fact}(F))\to Z^1(G,\xX).$$
\end{example}

\section{1-Cocycles as algebras over a monad}\selabel{2}
\setcounter{equation}{0}
\subsection*{$G$-Category structures as monoids}\hfill\break
Given a category $\cC$, let $\Fam(\cC)$ be the {\it category of families} of objects in $\cC$. Recall that the objects of $\Fam(\cC)$ are all families $(A_i)_{i\in I}$ with $A_i\in \cC$,
and a morphism $(A_i)_{i\in I}\to (B_j)_{j\in J}$ 
is a pair $(f,u)$ in which $f:\ I\to J$ is a map of sets, and $u$ is a family of morphisms
$u_i :\ A_i \to B_{f(i)}$. If $\cC$ were a monoidal category
$\cC=(\cC,\ot,1,\alpha,\lambda,\rho)$, then so is $\Fam(\cC)=
(\Fam(\cC),\ot,1,\alpha',\lambda',\rho')$ with
\begin{equation}\eqlabel{2.1}
(A_i)_{i\in I}\ot (B_j)_{j\in J}= (A_i\ot B_j)_{(i,j)\in I\times J},
\end{equation}
with $1$ in $\Fam(\cC)$ being the one-member family corresponding to the $1$ in $\cC$, and with obvious $\alpha',\lambda',\rho'$. According to this description, a {\it monoid} in 
$(\Fam(\cC),\ot,1)$ can be identified with a system $((M_i)_{i\in I},e,m)$, in which:\\
$\bullet$ $I$ is an ordinary monoid;\\
$\bullet$ $m$ is a family of morphisms $m_{(i,j)}:\ M_i\ot M_j\to M_{ij}$ in $\cC$ 
satisfying the suitable associativity condition;\\
$\bullet$ $e$ is a morphism in $\cC$ from $1$ to $M_1$ (where the index $1$ of $M_1$ is the identity element of the monoid $I$) satisfying the suitable identity condition.\\
In particular, if $\cC=(\cC,\ot,1)$ was a strict monoidal category, and 
$(M_i)_{i\in I}$ an object in $\Fam(\cC)$ in which\\
$\bullet$ $ I$ is equipped with (an ordinary) monoid structure;\\
$\bullet$ $M_i\ot M_j=M_{ij}$ for all $i,j\in I$;\\
$\bullet$ $M_1=1$,\\
then the system $((M_i)_{i\in I},e,m)$, in which all $m_{(i,j)}$ and $e$ are identity morphisms, is a monoid in $\Fam(\cC)$ (although $\Fam(\cC)$ is not strict!); let us call such a monoid {\it split}. Comparing this notion with the notion of $G$-category considered in \seref{1}, we immediately observe:

Let $G$ be a group. To make a category $\xX$ a $G$-category is the same thing as to give a split monoid $(M_i)_{i\in I}$ in the monoidal category $\Fam(\xX^{\xX})$, where the monoidal 
structure on 
the category $\xX^{\xX} = (\xX^{\xX},\circ, 1_{\xX})$ of endofunctors of $\xX$, with $I = G$. According to the notation of \seref{1}, (the underlying object of) that monoid is to be written as
\begin{equation}\eqlabel{2.2}
(g(-))_{g\in G}.
\end{equation}

\subsection*{Split monoid actions as $1$-cocycles}\hfill\break
Let $M = (M,e,m)$ be a monoid in a monoidal category $\cC$. 
Recall that an $M$-action in $\cC$ is a pair $(X,\xi)$, where $X$ is an object in $\cC$ and 
$\xi:\ M\ot X\to X$ a morphism in $\cC$ satisfying the associativity and identity properties with respect to the monoid structure of $M$. We will also say that $\xi$ is an $M$-action on $X$. Consider the following

\begin{situation}\silabel{2.1}
We take:\\
$\bullet$ $\cC$ to be a strict monoidal category, but then replace it by the monoidal category $\Fam(\cC)$ as defined above.\\
$\bullet$ $M=((M_i)_{i\in I},e,m)$ a split monoid in $\Fam(\cC)$.\\
$\bullet$ $X$ an object in $\cC$ regarded as a one-member family, and therefore an object in $\Fam(\cC)$.
In this case an $M$-action on $X$ can be described as a family
$\xi=(\xi_i)_{i\in I}$ of morphisms $\xi_i:\ M_i\ot X\to X$ with
\begin{equation}\eqlabel{2.3}
\xi_1=1_X~~{\rm and}~~\xi_{ij}=\xi_i(M_i\ot \xi_j).
\end{equation}
\end{situation}

After that consider a further special case:
\begin{situation}\silabel{2.2}
We take $(\cC,\ot 1)=(\xX^\xX,\circ,1_X)$ (where $\xX$ is a $G$-category as in \seref{1}),
$(M_i)_{i\in I}$ to be the monoid \equref{2.2} , and $X$ an object in $\xX$ regarded as a constant functor
$\xX\to \xX$. We will say that $(X,\xi)$ is a $(g(-))_{g\in G}$-action on an object in $\xX$.
\end{situation}

Comparing the formulae \equref{2.3} with those in \deref{1.1}(a), we obtain:

\begin{proposition}\prlabel{2.3}
A pair $(X,\xi)=(X,(\xi_g)_{g\in G})$ is an object in $Z^1(G,\xX)$ if and only if it is a
$(g(-))_{g\in G}$-action (on $X$). This yields a canonical isomorphism between the 
category  $Z^1(G,\xX)$ of 1-cocycles $G\to \xX$ and the category of $(g(-))_{g\in G}$-actions on objects in $\xX$.
\end{proposition}

\subsection*{Involving the coproduct functor $\Fam(\cC)\to \cC$}\hfill\break
In (say) \siref{2.1}, if $\cC$ had coproducts preserved by the functors
$C\ot -,~-\ot C:\ \cC\to \cC$ for every object $C$ in $\cC$, then the monoid structure of
$(M_i)_{i\in I}$ would of course induce a monoid structure on the coproduct
$\coprod_{i\in I} M_i$, and there would be a straightforward way to compare the actions of these two monoids. However having in mind \siref{2.2}, we cannot assume the preservation of coproducts by 
$C\ot -$ for every $C$; what we really need is the following:

\begin{proposition}\prlabel{2.4}
Let $\cC$ be a monoidal category with coproducts preserved by
$-\ot C:\ \cC\to \cC$ for each object $C$ in $\cC$, and $M=((M_i)_{i\in I},e,m)$
a monoid in $\Fam(\cC)$ such that $M_i\ot -:\ \cC\to \cC$ 
preserves coproducts for each $i\in I$. Then the monoid structure of 
$(M_i)_{i\in I}$ induces a monoid structure on the coproduct $\coprod_{i\in I} M_i$ and there is a canonical isomorphism between the category of $(M_i)_{i\in I}$-actions (i.e. $M$-actions) on one-member families in $\Fam(\cC)$ and the category of $\coprod_{i\in I} M_i$-actions in $\cC$.
\end{proposition}

In \siref{2.2}, assuming the existence of $G$-indexed coproducts in $\xX$, all the required conditions hold since all the functors $g(-)$ ($g \in G$) are isomorphisms and therefore preserve coproducts. Here $\coprod_{i\in I} M_i= \coprod_{g\in G}g(-)$ being a monoid in $(\xX^\xX,\circ,1_{\xX})$
 is a monad on $\xX$; let us denote it by $G(-)$. From Propositions \ref{pr:2.3} and \ref{pr:2.4}, we obtain

\begin{theorem}\thlabel{2.5}
Let $G$ be a group and $\xX$ a $G$-category with $G$-indexed coproducts. Then there is a canonical isomorphism between the category $Z^1(G,\xX)$ of 1-cocycles $G\to \xX$ and the
category $\xX^{G(-)}$ of $G(-)$-algebras. For a 1-cocycle $(X,\xi):\ G\to \xX$ with
$\xi=(\xi_g)_{g\in G}$, the corresponding  $G(-)$-algebra is the pair $(X,\xi_*)$
with the same $X$ and $\xi_*:\ G(X)=\coprod_{g\in G} g(X)\to X$ induced by all
$\xi_g:\ g(X)\to X$ ($g\in G$).
\end{theorem}

{\bf Convention}: we will identify $\xi$ and $\xi_*$, and just write $\xi$.

\section{$G$-Galois monads and adjunctions}\selabel{3}
\setcounter{equation}{0}
\subsection*{Relating split and connected monoids}\hfill\break
Let $M=((M_i)_{i\in I},e,m)$ be a split monoid as in \siref{2.1}, and
$M'=(M',e',m')$ a monoid in the monoidal category $\cC$. Regarding $\cC$ as a 
sub-monoidal-category in $\Fam(\cC)$, we can regard $M'$ as a monoid in $\Fam(\cC)$, and we will call such a monoid in $\Fam(\cC)$ {\it connected}. We need the description of monoid homomorphisms 
from $M$ to $M'$, i.e. morphisms $u:\ M\to M'$ in $\Fam(\cC)$ making the diagrams
\begin{equation}\eqlabel{3.1}
\xymatrix{
M\ar[rr]^{u}&&M'\\
&1\ar[lu]^e\ar[ru]_{e'}&}
\hspace*{16mm}
\xymatrix{M\ot M \ar[rr]^{u\ot u}\ar[d]_{m}&& M'\ot M'\ar[d]^{m'}\\
M\ar[rr]^{u}&& M'}
\end{equation}
commute. Since $M'$ regarded as an object in $\Fam(\cC)$ is a one-member family, a morphism 
$u:\ M\to M'$ in $\Fam(\cC)$ is nothing but an $I$-indexed family of morphisms 
$u_i:\ M_i\to M'$ - and the commutativity of \equref{3.1} expressed in the language of $u_i$'s becomes
\begin{equation}\eqlabel{3.2}
\xymatrix{
M_1\ar[rr]^{u_1}&&M'\\
&1\ar[lu]^=\ar[ru]_{e'}&}
\hspace*{16mm}
\xymatrix{M_i\ot M_j \ar[rr]^{u_i\ot u_j}\ar[d]_{=}&& M'\ot M'\ar[d]^{m'}\\
M_{ij}\ar[rr]^{u_{ij}}&& M'}
\end{equation}
i.e.
\begin{equation}\eqlabel{3.3}
e'=u_1,~~~m'(u_i\ot u_j)=u_{ij}~~(i,j\in I).
\end{equation}

We will use an example of such a homomorphism provided by

\begin{lemma}\lelabel{3.1}
Let $M=((M_i)_{i\in I},e,m)$ and $M'=(M',e',m')$ be as above, and let
$M'\ot M_i=M'$ and $m'\ot M_i=m'$ for every $i\in I$. Then the family
$u=(u_i)_{i\in I}$, where $u_i=e'\ot M_i$ is the morphism
\begin{equation}\eqlabel{3.4}
\xymatrix{
M_i=M_{1i}=M_1\ot M_i\ar[rr]^(.54){e'\ot M_i}&&M'\ot M_i=M',}
\end{equation}
is a monoid homomorphism from $M$ to $M'$.
\end{lemma}

\begin{proof}
We have that
$e'=e'\ot 1=e'\ot M_1=u_1$,
and
\begin{eqnarray*}
&&\hspace*{-2cm}
m'(u_i\ot u_j)=
m'(M'\ot u_j)(u_i\ot M_j)~~{\rm (since}~\ot~{\rm is~a~bifunctor)}\\
&=& (m'\ot M_j)(M'\ot e'\ot M_j)(e'\ot M_i\ot M_j),
\end{eqnarray*}
since $m'=m'\ot M_j$, $u_j=e'\ot M_j$ and $u_i=e'\ot M_i$.
\end{proof}

\leref{3.1} easily gives

\begin{corollary}\colabel{3.2}
Under the assumptions of \prref{2.4} and \leref{3.1}, the morphism
$\coprod_{i\in I} M_i\to M'$ induced by the family of morphisms 
\equref{3.4} is a monoid homomorphism in the monoidal category $\cC$.
\end{corollary}

Applying this to \siref{2.2}, we obtain

\begin{corollary}\colabel{3.3}
Let $(T,\eta,\mu)$ be a monad on $\xX$ with $T(g(-))=T$ and $\mu(g(-))=\mu$
for all $g\in G$. Then the natural transformation
$\gamma:\ G(-)=\coprod_{g\in G} g(-)\to T$ induced by the natural transformations
\begin{equation}\eqlabel{3.5}
g(-)=1_{\xX}(g(-))\mapright{\eta(g(-))}T(g(-))=T~~~(g\in G),
\end{equation}
is a morphism of monads.
\end{corollary}

\subsection*{$G$-Galois monads and adjunctions}\hfill\break
For a monad $T=(T,\eta,\mu)$ on $\xX$, let
\begin{equation}\eqlabel{3.6}
\xymatrix{
\xX\ar[r]<3pt>^{F^T}&\xX^T\ar[l]<3pt>^{U^T}},~~~
\eta:\ 1_{\xX}\to U^TF^T,~\varepsilon:\ F^TU^T\to 1_{\xX^T}
\end{equation}
be the associated free-forgetful adjunction. 
Then the assumptions $T(g(-)) = T$ and $\mu(g(-)) = \mu$ of \coref{3.3} 
can be reformulated as $F^T(g(-))=F^T$, and
$\gamma:\ G(-)=\coprod_{g\in G} g(-)\to T$
being a morphism of monads induces a functor
\begin{equation}\eqlabel{3.7}
K^{\gamma}:\ \xX^T\to \xX^{G(-)}\cong Z^1(G,\xX).
\end{equation}
Explicitly, for a $T$-algebra $(X,\zeta)$, we have
\begin{equation}\eqlabel{3.8}
K^{\gamma}(X,\zeta)=(X,\xi)~~{\rm with}~~\xi_g=\zeta\eta_{g(X)}~~~(g\in G).
\end{equation}

\begin{definition}\delabel{3.4}
Let $G$ be a group and $\xX$ a $G$-category with $G$-indexed coproducts. A monad 
$T = (T,\eta,\mu)$ on $\xX$ is said to be a $G$-Galois monad if
\begin{enumerate}
\item[(a)] $T(g(-)) = T$ and $\mu(g(-)) = \mu$ for all $g\in G$;
\item[(b)] the morphism $\gamma:\ G(-)=\coprod_{g\in G} g(-)\to T$
is an isomorphism.
\end{enumerate}
Equivalently, $T$ is a $G$-Galois monad if the free functor $F^T$ satisfies 
$F^T(g(-)) = F^T$ for all 
$g\in G$ and the functor $K^{\gamma}$ is an isomorphism.
\end{definition}

If instead of a monad $T$ we begin with an adjunction
\begin{equation}\eqlabel{3.9}
\xymatrix{
\xX\ar[r]<3pt>^{F}&\yY\ar[l]<3pt>^{U}},~~~
\eta:\ 1_{\xX}\to UF,~\varepsilon:\ FU\to 1_{\yY}
\end{equation}
where $\xX$ is a $G$-category again, \deref{3.4} should be reformulated as

\begin{definition}\delabel{3.5}
An adjunction \equref{3.9} is said to be a $G$-Galois adjunction if
\begin{enumerate}
\item[(a)] $F(g(-))=F$, for all $g\in G$;
\item[(b)] the category $\xX$ admits $G$-indexed coproducts and the morphism
$\gamma_X:\ \coprod_{g\in G} g(X)\to UF(X)$ induced by the collection of morphisms
$\eta_{g(X)}:\ g(X)\to UF(g(X))=UF(X)$ $(g\in G)$ 
is an isomorphism for each object $X\in \xX$;
\item[(c)] the functor $U$ is monadic.
\end{enumerate}
\end{definition}

Recall that for a general adjunction \equref{3.9} with the corresponding monad 
$T=(T,\eta,\mu)=(UF,\eta,U\varepsilon F)$,
the comparison functor $K:\ \yY\to \xX^T$ 
is defined by $K(Y) = (U(Y),U(\varepsilon_Y))$, and when $U$ is monadic the 
quasi-inverse $L:\  \xX^T\to \yY$ of $K$ can be described via the coequalizer diagram
\begin{equation}\eqlabel{3.10}
\xymatrix{
FUF(X) \ar[rr]<3pt>^{\varepsilon_{F(X)}} \ar[rr]<-3pt>_{F(\zeta)}&&F(X)
\ar[r]&L(X,\zeta)}
\end{equation}
where $(X,\zeta)$ is any $T$-algebra. Moreover, if $\yY$ admits coequalizers, then $K$ always (i.e. even under no monadicity assumption) has a left adjoint $L$ defined in the same way.

Let us provide another construction for the functor $L$ in the situation of \deref{3.5}, using the fact that the monad involving there is isomorphic (by \ref{de:3.5}(b)) to the monad $G(-)$.
 At the same time this will describe the composite of $L$ with the isomorphism between the categories $Z^1(G,\xX)$ and $\xX^T$, i.e. show how to calculate the object in $\yY$ corresponding to a given 1-cocycle $G\to \xX$. For, we observe:\\
$\bullet$ Since $\gamma_X:\ \coprod_{g\in G} g(X)\to UF(X)$ in \deref{3.5}
is an isomorphism, we can replace the two parallel arrows in \equref{3.10} by their composites with 
$F(\gamma_X)$. Moreover, since the functor $F$ being a left adjoint sends coproducts to coproducts, we can make a further replacement by composing with the canonical isomorphism
$\coprod_{g\in G} F(g(X))\to F(\coprod_{g\in G} g(X))$. After that, instead of defining
$L(X,\zeta)$ via the coequalizer diagram \equref{3.10}, 
we can define it via the morphism $F(X) \to L(X,\zeta)$ that composed with
\begin{equation}\eqlabel{3.11}
\xymatrix{
F(g(X))\ar[r]^{F(\iota_g)}& F(G(X))\ar[r]^{F(\gamma_X)}& FUF(X)\ar[r]^{\varepsilon_{F(X)}}&
F(X)}
\end{equation}
and with
\begin{equation}\eqlabel{3.12}
\xymatrix{
F(g(X))\ar[r]^{F(\iota_g)}& F(G(X))\ar[r]^{F(\gamma_X)}& FUF(X)\ar[r]^{F(\zeta)}&
F(X)}
\end{equation}
makes them equal for every $g\in G$, and is universal with this property (where 
$\iota_g$ is the coproduct injection corresponding to $g$).\\
$\bullet$ For the composite \equref{3.11}, we have
$$
\varepsilon_{F(X)} F(\gamma_X)F(\iota_g)=\varepsilon_{F(X)} F(\gamma_X\iota_g)
\equal{\ref{de:3.5}(b)}
\varepsilon_{F(X)}  F(\eta_{g(X)})
\equal{\ref{de:3.5}(a)}\varepsilon_{F(g(X))}F(\eta_{g(X)}),$$
which is the identity morphism of $F(g(X)) = F(X)$ by one of the triangular identities 
for adjoint functors.\\
$\bullet$ And for \equref{3.12}:
$$F(\zeta)F(\gamma_X)F(\iota_g)=
F(\zeta\gamma_X\iota_g)=f(\zeta\eta_{g(X)})=f(\xi_g),$$
where $\xi$ is as in \equref{3.8}.\\
$\bullet$ As easily follows from \ref{de:1.1}(a) and \ref{de:3.5}(a), the assignment 
$g\mapsto F(\xi_g)$ determines a group homomorphism 
$G \to \Aut(F(X))$, i.e. makes $F(X)$ a $G$-object in $\yY$. Moreover, if we regard $\yY$ as a 
$G$-category with the trivial action of $G$ then the category $Z^1(G,\yY)$ of 1-cocycles 
$G\to \yY$ is of course nothing but the category $\yY^G$ of $G$-objects in $\yY$, and then 
$F$, obviously being a $G$-functor, induces a functor
\begin{equation}\eqlabel{3.13}
Z^1(G,F):\ Z^1(G,\xX)\to Z^1(G,\yY)=\yY^G.
\end{equation}
And of course the $G$-object in $\yY$ obtained above is nothing but
$Z^1(G,F)(X,\xi)$.\\
$\bullet$ If $\yY$ admits ``allÓ colimits (say, as large as its hom-sets), then there is the colimit functor
\begin{equation}\eqlabel{3.14}
{\rm colim}:\ \yY^G\to \yY
\end{equation}
which sends a $G$-object $Y$ to its {\it universal quotient with the trivial} $G$-{\it action}. We will briefly write ${\rm colim}(Y) = Y/G$, also when that colimit exists just for a given object $Y$. In particular 
$L(X,\zeta)=F(X)/G$.

Summarizing, and using the isomorphism \equref{3.7} defined by \equref{3.8} and the definition of the comparison functor $K : \yY\to \xX^T$, we obtain

\begin{theorem}\thlabel{3.8}
(a) For an adjunction \equref{3.9} satisfying \ref{de:3.5}(a) and (b) the left adjoint 
$L:\ \xX^T\to \yY$ of the comparison functor $K:\ \yY\to \xX^T$ can be described as
\begin{equation}\eqlabel{3.15}
L(X,\zeta)=F(X)/G,
\end{equation}
where $F(X)$ is regarded as a $G$-object in $\yY$ via
$$g\mapsto F(\xi_g)=F(\zeta\eta_{g(X)}):\ F(X)\to F(X).$$
(b) If \equref{3.9} is $G$-Galois adjunction then there is a category equivalence
\begin{equation}\eqlabel{3.16}
\yY\sim Z^1(G,\xX),
\end{equation}
under which an object $Y$ in $\yY$ corresponds to the 1-cocycle
\begin{equation}\eqlabel{3.17}
\bigl(U(Y), (U(\varepsilon_Y)\eta_{g(U(Y))})_{g\in G}\bigr):\ G\to \xX;
\end{equation}
conversely, a 1-cocycle $(X,\xi):\ G\to \xX$ corresponds to the object $F(X)/G$,
where $F(X)$ is regarded as a $G$-object in $\yY$ via $g\mapsto F(\xi_g):\ 
F(X)\to F(X)$.
\end{theorem}

Furthermore, applying \prref{1.4}, we obtain:

\begin{theorem}\thlabel{3.9}
If \equref{3.9} is a $G$-Galois adjunction, 
then every object $Y_0$ in $\yY$ admits a category equivalence
\begin{equation}\eqlabel{3.18}
\yY[Y_0]\sim Z^1(G,\End(U(Y_0))_{\xi_0}),
\end{equation}
in which $\yY[Y_0]$ is the full subcategory in $\yY$ with objects all $Y$ in $\yY$ with $U(Y)$ isomorphic to $U(Y_0)$, and 
$(U(Y_0),\xi_0)=(U(Y_0), (U(\varepsilon_{Y_0})\eta_{g(U(Y_0))})_{g\in G})$ is the
1-cocycle $G\to \xX$ corresponding to $Y_0$ under the equivalence \equref{3.16}.
In particular the set of isomorphism classes of objects in  $\yY[Y_0]$ is bijective to
$H^1(G,\End(U(Y_0))_{\xi_0})= H^1(G,\Aut(U(Y_0))_{\xi_0})$.
\end{theorem}

\begin{remark}\relabel{3.10}
The equivalence \equref{3.18} is ``less canonical" than the equivalence \equref{3.16} in the following sense: let $\yY[Y_0]$ be the full subcategory in $\yY$ with objects all $Y$ in $\yY$ with 
$U(Y) = U(Y_0)$, and let
\begin{equation}\eqlabel{3.19}
\yY(Y_0)\to Z^1(G,\xX)_{U(Y_0)}~~{\rm and}~~
Z^1(G,\xX)_{U(Y_0)} \to \yY[Y_0]
\end{equation}
be the functors induced by the equivalence \equref{3.16}. Since both of them obviously are category equivalences, so are the composites
\begin{equation}\eqlabel{3.20}
\begin{array}{c}
\yY(Y_0)\to Z^1(G,\xX)_{U(Y_0)}\to Z^1(G, \End(U(Y_0))_{\xi_0})~~{\rm and}\\
Z^1(G, \End(U(Y_0))_{\xi_0})\to Z^1(G,\xX)_{U(Y_0)} \to \yY[Y_0],
\end{array}
\end{equation}
where $\xi_0$ is as in \thref{3.9}, and the isomorphism
$Z^1(G,\xX)_{U(Y_0)}\cong Z^1(G, \End(U(Y_0))_{\xi_0})$
is obtained from \prref{1.4}. However, 
since the category $\yY(Y_0)$ is only equivalent, but not equal to $\yY[Y_0]$, the functors 
\equref{3.20} are not exactly quasi-inverse to each other. One can say that the equivalence 
\equref{3.18} is only canonical up to a choice of the quasi-inverse of the inclusion functor 
$\yY(Y_0)\to  \yY[Y_0]$. Yet, since the second functor in \equref{3.20} is uniquely determined only up to the choice of colimits in $\yY$, it is useful to describe it explicitly. And, since it is a composite of explicitly given functors, the description is straightforward:
\begin{equation}\eqlabel{3.21}
(\varphi\in Z^1(G,\End(U(Y_0))_{\xi_0}))\mapsto F(Y_0)/G,
\end{equation}
where $G$ acts on $F(Y_0)$ by
\begin{equation}\eqlabel{3.22}
g\mapsto F(\varphi(g)U(\varepsilon_{Y_0}\eta_{g(U(Y_0))}),
\end{equation}
since $g\mapsto F(\xi_g)$ in \equref{3.15}, the role of $\xi_g$ is played by $\varphi(g)\xi_{0g}$
here according to \prref{1.4}, and $\xi_{0g}=U(\varepsilon_{Y_0}\eta_{g(U(Y_0))})$
as in \thref{3.9}.
\end{remark}

\section{Galois descent for modules and algebras}\selabel{4}
\setcounter{equation}{0}
\subsection*{Galois adjunctions correspond to finite Galois extensions}\hfill\break
Let us take the adjunction \equref{3.9} to be
\begin{equation}\eqlabel{4.1}
\xymatrix{
(E\hbox{-}\Mod)^{\rm op}
\ar[rr]<3pt>^{F=(-)_p}&&(B\hbox{-}\Mod)^{\rm op}\ar[ll]<3pt>^{U=E\ot_B(-)}},
\end{equation}
$$
\eta:\ 1_{(E\hbox{-}\Mod)^{\rm op}}\to UF,~\varepsilon:\ FU\to 
1_{(B\hbox{-}\Mod)^{\rm op}},$$
where\\
$\bullet$ $p:\ B\to E$ is a homomorphism of non-trivial commutative rings with 1;\\
$\bullet$ $F:\ (E\hbox{-}\Mod)^{\rm op}\to (B\hbox{-}\Mod)^{\rm op}$ 
is the (dual of) the restriction-of-scalars functor from the dual (=opposite) category of $E$-modules to the dual category of $B$-modules, for which we shall write
$F(D)=D_p$;\\
$\bullet$ $U:\ (B\hbox{-}\Mod)^{\rm op}\to (E\hbox{-}\Mod)^{\rm op}$ 
is the (dual of) the induction functor, for which we shall write
$U(A)=E\ot_B A$;\\
$\bullet$ $\eta=(\eta_D:\ E\ot_B D_p\to D_p)_{D\in (E\hbox{-}\Mod)^{\rm op}}$ is defined by
$\eta_D(e\ot d)=ed$ (here and below, whenever it makes sense, we keep the direction of arrows as in the categories of modules, not as in their duals);\\
$\bullet$ $\varepsilon=\varepsilon_A:\ A\to (E\ot_B A)_p)_{A\in (B\hbox{-}\Mod)^{\rm op}}$ is defined by
$\varepsilon_A(a)=1\ot a$.

We will also assume that $(E,p)$ is equipped with a right action of a group $G$, i.e. 
$G$ acts on $E$ on the right (via ring automorphisms of $E$) with $p(b)g = p(b)$ for each 
$b\in B$ and $g\in G$. This makes $(E\hbox{-}\Mod)^{\rm op}$ a $G$-category, 
in which $g(D) = D_g$, i.e. $g(D) = D$ as abelian groups and, for $e\in E$ and 
$d\in D$, we have:
\begin{equation}\eqlabel{4.2}
ed\in g(D)~~{\rm is~the~same~as}~~(eg)d\in D.
\end{equation}

\begin{theorem}\thlabel{4.1}
The adjunction \equref{4.1} equipped with the $G$-category structure on $(E\hbox{-}\Mod)^{\rm op}$
 defined by \equref{4.2} is a $G$-Galois adjunction if and only if the following conditions hold:
\begin{enumerate}
\item[(a)] $p:\ B\to E$ is an effective (co)descent morphism, i.e., considered as a homomorphism of B-modules, it is a pure monomorphism;
\item[(b)]  the map $h:\ E\ot_B E\to GE$, where $GE$ denotes the $E$-algebra of all maps from
$G$ to $E$, defined by $h(e\ot e')(g)=(eg)e'$, is bijective;
\item[(c)] the group $G$ is finite.
\end{enumerate}
\end{theorem}

\begin{proof}
(a) simply means that the functor $U:\ (B\hbox{-}\Mod)^{\rm op}\to
(E\hbox{-}\Mod)^{\rm op}$ is monadic (see \cite{JT} for details). On the other 
hand condition \ref{de:3.5} (b) now becomes: the map
\begin{equation}\eqlabel{4.3}
\gamma_D:\ E\ot_B D_p\to GD,~~\gamma_D(e\ot d)(g)=(eg)d,
\end{equation}
is bijective for each $D$ in $E\hbox{-}\Mod$.\\
Since this condition implies \ref{th:4.1}(b), it suffices to prove that, under conditions 
\ref{th:4.1} (a) and \ref{th:4.1}(b), condition \ref{de:3.5}(b) is equivalent to condition \ref{th:4.1}(c). For, note that we can identify $h$ with $\gamma_E$, and consider the commutative diagram
\begin{equation}\eqlabel{4.4}
\xymatrix{
(E\ot_B E)\ot_E D_p\ar[rr]^{h\ot 1_D}\ar[d]&&(GE)\ot_E D_p\ar[d]\\
E\ot_B D_p\ar[rr]^{\gamma_D}&&GD,}
\end{equation}
whose vertical arrows are obvious canonical maps. Since its top arrow is bijective by 
\ref{th:4.1} (b), and its left-hand vertical arrow is always bijective, the bijectivity of its bottom arrow is equivalent to the bijectivity of its right-hand vertical arrow. Therefore all we need to observe is that the canonical map $(GE)\ot_E D_p \to GD$ is bijective for every 
$E$-module $D$ if and only if $G$ is finite.
\end{proof}

\begin{remark}\relabel{4.2}
One can replace ordinary modules with modules equipped with various kinds of additional structure, and, in particular, with various kinds of algebras. To be more precise, \thref{4.1} still holds if instead of 
$E\hbox{-}\Mod$ and $B\hbox{-}\Mod$ we had $E\hbox{-}\Alg$ and 
$B\hbox{-}\Alg$ respectively, where ``$\hbox{-}\Alg$" refers to e.g.:
(a) arbitrary (not necessarily associative or commutative) algebras, with or without 1;

(b)  associative algebras, with or without 1;

(c)  (associative and) commutative algebras, with or without 1;

(d)  Lie algebras;

(e)  Jordan algebras;

(f)  differential algebras.\\
All we need here is to know that the functors $F$ and $U$ of \equref{4.1} induce similar functors for these kinds of algebras, and that replacing modules with algebras does not affect monadicity. For, see Section 6 in \cite{JT}.
\end{remark}

\subsection*{Translation of Theorems \ref{th:3.8} and \ref{th:3.9} for modules and algebras}\hfill\break
\thref{4.1} allows to apply Theorems \ref{th:3.8} and \ref{th:3.9} to modules and algebras. This gives:

\begin{theorem}\thlabel{4.3}
Suppose the equivalent conditions of \thref{4.1} hold. Then the category of $B$-modules is equivalent to the category $Z^1(G,(E\hbox{-}\Mod)^{\rm op})^{\rm op}$, which can be identified with the category of 
$E$-modules $D$ on which $G$ acts on the right via $B$-module automorphisms with 
$(ed)g = (eg)(dg)$ for all $e\in E$, $d\in D$, and $g\in G$. Under this equivalence:\\
(a) a $B$-module $A$ corresponds to the $E$-module $E\ot_B A$, on which $G$ acts on the right via 
$B$-module automorphisms with $(e\ot a)g = eg\ot a$ for all $e\in E$, $a\in A$, and $g\in G$; \\
(b)  an $E$-module $D$ equipped with a right $G$-action as above corresponds to the $B$-module 
$D^G = \{d \in D~|~dg = d {\rm~for ~all~} g \in G\}$, where the scalar multiplication 
$B\times D^G\to D^G$ is induced by the scalar multiplication
$B\times D_p\to D_p$.
\end{theorem}

\begin{proof}
A $1$-cocycle $(D,\xi)\ : G \to (E\hbox{-}\Mod)^{\rm op}$ is an $E$-module $D$ together with a family 
$\xi=(\xi_g)_{g\in G}$ of $E$-module isomorphisms $\xi_g:\ D \to D_g$ in $\xX$ satisfying 
$\xi_1=1_X$ and 
$\xi_{hg} = \xi_g\xi_h$ (having in mind that $h(\xi_g)=\xi_g$, since 
$h(-) :\ (E\hbox{-}\Mod)^{\rm op}\to (E\hbox{-}\Mod)^{\rm op}$ is the restriction-of-scalars functor corresponding to $h$). Since the $E$-module structure on $D_g$ is given by \equref{4.2}, it is easy to 
see that to say that $\xi_g:\ D\to D_g$ is an $E$-module isomorphism is to say that 
$\xi_g(ed) = (eg)\xi_g(d)$ for all $e\in E$, $d\in D$, and $g\in G$. After that we put 
$dg = \xi_g(d)$, and this gives the desired description of the category 
$Z^1(G,(E\hbox{-}\Mod)^{\rm op})^{\rm op}$. It remains to prove (a) and (b).

(a) By \thref{3.8}(b), the $1$-cocycle $G \to (E\hbox{-}\Mod)^{\rm op}$ corresponding to a $B$-module $A$, 
is the pair $(E\ot_BA, (\eta_{g(E\ot_BA))}(E\ot_B\varepsilon_A))_{g\in G})$. Using the explicit formulae for $\varepsilon$ and $\eta$ given in the description of the adjunction \equref{4.1}, we have
$$(\eta_{g(E\ot_B A))}(E\ot_B \varepsilon_A))(e\ot a)=
\eta_{g(E\ot_B A))}(e\ot (1\ot a))=e(1\ot a)=eg\ot a,$$
where the last equality follows from the fact that $e(1\ot a)$ is calculated in
$g(E\ot_B A)$.
According to our way of defining a right G-action out of a $1$-cocycle
$G\to (E\hbox{-}\Mod)^{\rm op}$ this gives
$$(e\ot a)g= (\eta_{g(E\ot_B A))}(E\ot_B \varepsilon_A))(e\ot a)=eg\ot a,$$
as desired.

(b) now follows immediately from \thref{3.8}(b).
\end{proof}

\begin{theorem}\thlabel{4.4}
Suppose the equivalent conditions of \thref{4.1} hold, and $A_0$ is a fixed $B$-module.
Then the full subcategory of the category of $B$-modules
$$B\hbox{-}\Mod [A_0]=\{A\in B\hbox{-}\Mod~|~
E\ot_B A\cong E\ot_B A_{[0]}~{\rm in}~E\hbox{-}\Mod\}$$
 is equivalent to the category
$Z^1(G,{\End(E\ot_B A_0)^{\rm op}}_{\xi_0})^{\rm op}$, where
$(E\ot_B A_0,\xi_0):\ G \to (E\hbox{-}\Mod)^{\rm op}$ is the $1$-cocycle, corresponding to
$A_0$, which makes $\End(E\ot_B A_0)^{\rm op}$ a one-object $G$-category, as in
Observation \ref{obs1.3}. Moreover:

(a) The $G$-category constructed above on $\End(E\ot_B A_0)^{\rm op}$, written as
$$G\times \End(E\ot_B A_0)^{\rm op}\to \End(E\ot_B A_0)^{\rm op},~~~
(g,u)\mapsto gu,$$
has, for each $g\in G$, the $E$-module isomorphism
$g(u):\ E\ot_BA_0\to E\ot_BA_0$ defined by
\begin{equation}\eqlabel{4.5}
(gu)(t)= (u(tg))g^{-1}~~~~(t\in E\ot_B A_0),
\end{equation}
using the right action of $G$ on $E\ot_B A_0$ described in \ref{th:4.3} (a).

(b) In particular
\begin{equation}\eqlabel{4.6}
B\hbox{-}\Mod [A_0]/\cong ~{\rm is~bijective~to~}H^1(G,\Aut(E\ot_B A_0)^{\rm op},
\end{equation}
where $B\hbox{-}\Mod [A_0]/\cong$ is the set of isomorphism classes of objects of the category 
$B\hbox{-}\Mod [A_0]$ and the action of $G$ on $\Aut(E\ot_B A_0)^{\rm op}$ is defined by 
\equref{4.5}.
\end{theorem}

\begin{proof}
The first assertion of the theorem is really direct translation of \thref{3.9} and the same is true for the assertion \ref{th:4.4}(b) provided we can use \ref{th:4.4}(a). Therefore we only need to prove 
\ref{th:4.4}(a). For, we observe:

(i) As we see from the proof of \ref{th:4.3} (a), $\xi_0=(\xi_{0g})_{g\in G}$, where
$\xi_{0g}:\ E\ot_B A_0\to (E\ot_B A_0)_g$ is defined by $\xi_{0g}(e\ot a)=eg\ot a$.

(ii) Recall the formula $ga=\xi_g g(a)(\xi_g)^{-1}$ from Observation \ref{obs1.3}.
It tells us that $gu$ in \ref{th:4.4} (a) is defined by
$gu=\xi_{0g}u(\xi_{0g})^{-1}$ (since $g(-):\ (E\hbox{-}\Mod)^{\rm op}\to (E\hbox{-}\Mod)^{\rm op}$
is the dual restriction-of-scalars functor, we write $u$ instead of $g(u)$).

(iii) $eg\ot a)$ above is the same as $(e\ot a)g$ in $E\ot_B A_0$ (on which $G$ acts via $B$-module automorphisms as described in \ref{th:4.3}(a)).

\equref{4.5} follows from these three observations.
\end{proof}

\begin{remark}\relabel{4.5}
The following obvious reformulations are useful:

(a) Formula \equref{4.5} implies
\begin{equation}\eqlabel{4.7}
(gu)(1\ot a)=(u(1\ot a))g^{-1},
\end{equation}
for each $a\in A_0$ - which determines $gu$ since it is an $E$-module homomorphism and 
$E\ot_B A_0$ is generated by the elements of the form $1\ot a$ ($a\in A_0$)
as an $E$-module.

(b) Since sending elements to their inverses determines an isomorphism between
$\Aut (E\ot_B A_0)^{\rm op}$ and $\Aut(E\ot_B A_0)$, we can rewrite \ref{th:4.4}(b) as
\begin{equation}\eqlabel{4.8}
B\hbox{-}\Mod [A_0]/\cong ~{\rm is~bijective~to~}H^1(G,\Aut(E\ot_B A_0),
\end{equation}
where the action of $G$ on $\Aut(E\ot_B A_0)$ is defined by \equref{4.5} again.
\end{remark}

\begin{remark}\relabel{4.6}
If we are not interested in the cohomological side of the story, the first assertion of \thref{4.3}
 reduces to:

(a) Under the equivalent conditions of \thref{4.1}, the category $B\hbox{-}\Mod$
  is equivalent to the category of $E$-modules $D$ on which $G$ acts on the right via $B$-module automorphisms with $(ed)g = (eg)(dg)$ for all $e \in E$, $d \in D$, and $g \in G$.
  
After that \ref{th:4.3}(a) implies that:

(b) Under the equivalent conditions of \thref{4.1}, the category $B\hbox{-}\Mod [A_0]$
 is equivalent to the category defined as follows:\\
$\bullet$
the objects are right actions of $G$ on $E\ot_B A_0$ via $B$-module automorphisms with
\begin{equation}\eqlabel{4.9}
(et)g=(eg)(tg),~{\rm for~all~}e\in E,~t\in E\ot_B A_0,~{\rm and}~g\in G\};
\end{equation}
$\bullet$  the morphisms are $E$-module homomorphisms that preserve $G$-action.

This suggests to ask if \thref{4.4} provides a better (non-cohomological) description of the category 
$B\hbox{-}\Mod [A_0]$? It is easy to check that the direct translation of the first two assertions of 
\thref{4.4} gives:

(c) Under the equivalent conditions of \thref{4.1}, the category $B\hbox{-}\Mod [A_0]$ is equivalent to the category defined as follows:\\
$\bullet$ the objects are right actions $*$ of $G$ on $E\ot_B A_0$ via $E$-module automorphisms, which together with the fixed right $G$-action defined as in \ref{th:4.3}(a), satisfy
\begin{equation}\eqlabel{4.10}
t*hg=((t*h)h*g)h^{-1}~~{\rm for all~}t\in E\ot_B A_0,~{\rm and}~g,h\in G;
\end{equation}
$\bullet$ the morphisms are the $E$-module homomorphisms that preserve the $G$-action $*$.

Indeed, using \equref{1.4} with \equref{1.5}, and denoting $\varphi(g)(t)$ by $t*g$, we have
\begin{eqnarray*}
&&\hspace*{-2cm}
t*hg=\varphi(hg)(t)=((h\varphi(g))(\varphi(h)))(t)=(h\varphi(g))((\varphi(h)(t))\\
&=&(h\varphi(g))(t* h)= ((\varphi(g)((t* h)h)))h^{-1}= ((t* h)h* g)h^{-1},
\end{eqnarray*}
where the reason of writing $(h\varphi(g))(\varphi(h))$ instead of 
$(\varphi(h))(h\varphi(g))$ is that we are applying \equref{1.4} to $1$-cocycles
$G\to \End(E\ot_B A_0)^{\rm op}$.\\
In comparison with (b), the advantage of (c) is that the variable $G$-actions are actions via $E$-module automorphisms, although formula \equref{4.10} looks complicated unless it is interpreted cohomologically.
\end{remark}

{\it Again, all these calculations and results can be copied for various kinds of algebras (see 
\reref{4.2}).}

\subsection*{Galois cohomology of commutative rings}
\begin{theorem}\thlabel{4.7}
In the notation above, suppose conditions \ref{th:4.1}(b) and \ref{th:4.1}(c) hold. Then the following conditions are equivalent:
\begin{enumerate}
\item[(a)] condition \ref{th:4.1}(a);
\item[(b)] $p$ is injective and makes $E$ a finitely generated projective $B$-module;
\item[(c)] $p$ induces an isomorphism $B\cong E^G$.
\end{enumerate}
\end{theorem}

\begin{proof}
$(a)\Rightarrow (c)$: 
Under the category equivalence described in \thref{4.3}, the $B$-module $E$ corresponds to the 
$E$-module $E\ot_B B\cong E$, on which $G$ acts as originally on $E$ (see 
\ref{th:4.3}(a)), and then \ref{th:4.7}(c) follows from \ref{th:4.3}(b).\\
$(c)\Rightarrow (b)$ follows from the results of \cite{CHR} (see also 
\cite{CS} and \cite{DI}), and $(b)\Rightarrow (a)$ is simply an instance of the well-known implication
\begin{equation}\eqlabel{4.11}
{\rm ``finitely~generated~projective"}~~\Longrightarrow~~{\rm ``pure"}.\vspace*{-6mm}
\end{equation}
\end{proof}

From the results of \cite{CHR} we also conclude:

\begin{corollary}\colabel{4.8}
The equivalent conditions of \thref{4.1} hold if and only if $p:\ B\to E$ is a Galois extension with Galois group $G$ in the sense of S. U. Chase, D. K. Harrison, and A. Rosenberg \cite{CHR}. 
\end{corollary}

\begin{remark}\relabel{4.9}
Recall that when $E$ (and therefore also $B$) has no non-trivial idempotents, the fact that 
$p:\ B\to E$ is a Galois extension with Galois group $G$, implies that the action of $G$ induces an isomorphism $G \cong (\Aut_B(E))^{\rm op}$; in particular this is the case when $E$ (and therefore also $B$) is a field. Since every group is isomorphic to its opposite group, ``op" can be omitted here as well as in various applications of Theorems \ref{th:4.3} and \ref{th:4.4} (see in particular \reref{4.5}(b)). Note also that, if $p:\ B\to E$ is a Galois extension with Galois group $G$, merely the existence of an isomorphism $G \cong \Aut_B(E)$ implies that $E$ has no non-trivial idempotents. Therefore in the situation $G \cong \Aut_B(E)$, one usually can think of a Galois extension as simply a pair $(E,B)$, usually written as $E/B$, where $B \subseteq E$, $p:\ B\to E$ is the inclusion map, and 
$G = \Aut_B(E)$ (or, better, $G = (\Aut_B(E))^{\rm op}$). Moreover, in the 
case of fields, according to classical Galois theory, the following conditions on $B \subseteq E$ are equivalent:
\begin{enumerate}
\item[(a)]  $E/B$ is a Galois extension as in \coref{4.8}, with some $G$;
\item[(b)] $E/B$ is a Galois extension as in \coref{4.8}, with $G = (\Aut_B(E))^{\rm op}$;  
\item[(c)]  $B = E^G$ for some finite group $G$ that acts on $E$ via $B$-algebra automorphisms;  
\item[(d)]  $E/B$ is a finite normal and separable field extension;  
\item[(e)]  there exists a separable polynomial $f\in B[x]$, for which $E$ is the splitting field.
\end{enumerate}
And when these conditions hold, one often writes $H^n(E/B,-)$ instead of $H^n(G,-)$, or 
$H^n(E/B,F)$ instead of $H^n(G,F(E/B))$, where $F$ is a functor from a suitable category of extensions to the category of groups; the pointed sets $H^n(E/B,F)$ are then called Galois cohomology sets 
of $E/B$ in coefficients in $F$. In fact $H^n(E/B,F)$ is:
\begin{itemize}
\item always a group when $n=0$; 
\item just a pointed set and usually defined only for $n = 0, 1$, if the group $F(E/B)$ is not abelian; 
\item always an abelian group when so is $F(E/B)$.
\end{itemize}
The definition of Galois cohomology also extends to infinite extensions, and in particular $H^n(E/B,F)$ is sometimes written as $H^n(B,F)$ when $E$ is the separable closure of $B$. However, this involves cohomology of profinite (topological) groups, which we do not consider in this paper.
\end{remark}

\subsection*{Hilbert Theorem 90}\hfill\break
Let $p:\ B\to E$ be a Galois extension with Galois group $G$, and let us take 
$A_0$ to be the free $B$-module $B^n$ on $n$ generators. Using the standard isomorphisms
$E\ot_B B^n\cong E^n$ and
$\Aut_{E\hbox{-}\Mod}(E^n)\cong \GL_n(E)$, and \thref{4.4}, we obtain
\begin{equation}\eqlabel{4.12}
B\hbox{-}\Mod [B^n]/\cong~{\rm is~bijective~to~}H^1(G,\GL_n(E)),
\end{equation}
where $\GL_n(E)$ is the group of invertible $n\times n$-matrices with entries in $E$. 
But how does $G$ act on $\GL_n(E)$? Let us calculate this carefully.

First, according to standard linear algebra, an (invertible) $n\times n$-matrix $m = (m_{ij}) 
\in \GL_n(E)$ corresponds to the automorphism $\theta_M\in \Aut_{E\hbox{-}\Mod}(E^n)$
defined on basic vectors
$x_1=(1,0,\cdots,0),\cdots, x_n=(0,\cdots,0,1)$ by
$$\theta_M(x_i)=m_{1i}x_1+\cdots +m_{ni}x_n.$$
Next, identifying $E\ot_B B^n$ with $E^n$, we can assume that the basic vectors above are identified with $1\ot (1,0,\cdots,0),\cdots, 1\ot (0,\cdots,0,1)$ respectively, 
and for these basic vectors we can use \equref{4.7}, with $\theta_M$ playing the role of $u$. 
This gives:
\begin{eqnarray*}
(g\theta_M)(x_i) &=& (\theta_M(x_i))g^{-1}=
(m_{1i}x_1+\cdots +m_{ni}x_n)^{-1}\\
&=& (m_{1i}g^{-1})(x_1g^{-1})+\cdots+ (m_{ni}g^{-1})(x_ng^{-1})\\
&=& (m_{1i}g^{-1})x_1+\cdots (m_{ni}g^{-1})x_n,
\end{eqnarray*}
and therefore tells us that the matrix corresponding to $\theta_M$ must be 
$(m_{ij}g^{-1})$. That is, the action of $G$ on $\GL_n(E)$ can be defined by
\begin{equation}\eqlabel{4.13}
g(m_{ij})= (m_{ij}g^{-1}).
\end{equation}
All $B$-modules in $B\hbox{-}\Mod[B^n]$ are obviously isomorphic to each other when $B$ is a field. Therefore, if $B$ is a field, we have
\begin{equation}\eqlabel{4.14}
H^1(G,\GL_n(E))=0,
\end{equation}
i.e. $H^1(G,\GL_n(E))$ is a one-element set; for $n = 1$ this gives the classical result called ``Hilbert Theorem 90Ó:
\begin{equation}\eqlabel{4.15}
H^1(E/B,\uU)=H^1(G,\uU(E))=
H^1(G,\GL_1(E))=0,
\end{equation}
where $\uU(E)$ 
is the group of non-zero elements in $E$. One can also easily deduce its ``commutative-ring-versionÓ, which is the exact sequence
\begin{equation}\eqlabel{4.16}
0 \longrightarrow H^1(G,\uU(E)) \longrightarrow \Pic(B)\longrightarrow \Pic(E)
\end{equation}
of abelian groups; $\uU(E)$ here is the group of invertible elements in $E$, while $\Pic(B)$ and 
$\Pic(E)$ denote the Picard groups of $B$ and $E$ respectively.

\subsection*{Commutative separable algebras}\hfill\break
Let $B\subseteq A = B(a)$ be a field extension generated, over $B$, by a single separable element $a$ (which is always the case for a finite separable field extension), $f\in B[x]$ the minimal polynomial of $a$ over $K$, $E$ a field containing the splitting field of $f$, and 
$e_1,\cdots,e_n$ the roots of $f$ in $E$, yielding
\begin{equation}\eqlabel{4.17}
f=\prod_{i=1,\cdots, n} (x-e_i)~~{\rm in}~~E[x].
\end{equation}
Having in mind that $A \cong B[x]/fB[x]$ and that the linear polynomials $x - e_i$ ($i = 1,\cdots,n$) are relatively prime, we calculate
\begin{eqnarray*}
&&\hspace*{-1cm}
E\ot_BA \cong E\ot_B (B[x]/fB[x])\cong E[x]/fE[x]\\
&\cong&
E[x]/(\prod_{i=1,\cdots, n} (x-e_i)) E[x]
\cong \prod_{i=1,\cdots,n} (E[x]/(x-e_i)E[x])\cong E^n,
\end{eqnarray*}
which suggests the following

\begin{definition}\delabel{4.10}
Suppose $E/B$ is a (finite) Galois field extension; a $B$-algebra $A$ is said to be split over 
$E$, if $E\ot_BA \cong E^n$ as $E$-algebras for some natural number $n$.
\end{definition}

In fact it is well known that:\\
$\bullet$ a commutative $B$-algebra is {\it separable} if and only if it is split over some finite Galois 
extension of $B$, and, equivalently, if and only if it is a finite product of subextensions 
of some finite Galois extension of $B$;\\
$\bullet$  therefore a field extension of $B$ is a separable $B$-algebra if and only if it is a finite separable extension.

Moreover, there are various generalizations of separability that reasonably agree with suitable generalized versions of \deref{4.10} (see \cite{BJ} and references there, especially 
\cite{DI,Gro,M}). Returning to the case of a fixed Galois field extension $E/B$, we have
that $B\hbox{-}\Alg[B^n]$ is  the category of $B$-algebras that are $n$-dimensional as
           $B$-vector spaces, and are finite products of subextensions of $E/B$, and
\begin{equation}\eqlabel{4.19}
B\hbox{-}\Alg[B^n]/\cong~{\rm is~bijective~to~}H^1(G,\sS_n),
\end{equation}
since $\Aut_{E\hbox{-}\Alg}(E\ot_B B^n)\cong \Aut_{E\hbox{-}\Alg}(E^n)
\cong \sS_n$, the $n$-th symmetric group. Here, the $G$-action on $\sS_n$ is trivial since:\\
$\bullet$ the $E$-algebra $E\ot_B B^n$ is (again, not freely though) generated by the set
\begin{equation}\eqlabel{4.20}
X(E\ot_B B^n)=\{1\ot (1,0,\cdots,0),\cdots, 1\ot (0,\cdots,0,1)\}
\end{equation}
of its minimal (non-zero) idempotents, which is invariant under any ring automorphism;\\
$\bullet$ 
 for every $a\in B^n$, we have $(1\ot a)g = 1\ot a$ by \ref{th:4.3}(a), and so \equref{4.5}
  gives $(gu)(t) = u(t)$whenever $t\in X(E\ot_B B^n)$.
  
\subsection*{Azumaya algebras}\hfill\break
Let us restrict our considerations again to the case of a Galois field extension. While Azumaya algebras over a field can simply be defined as central simple algebras, let us use a definition that generalizes to the ring case: 

\begin{definition}\delabel{4.11} 
Suppose $E/B$ is a (finite) Galois field extension; a (non-commutative)
$B$-algebra $A$ is said to be an Azumaya algebra split over $E$, if 
$E\ot_BA \cong \mM_n(E)$ as 
$E$-algebras for some natural number $n$, where $\mM_n(E)$ is the $E$-algebra of $n\times n$-matrices with entries in $E$.
\end{definition}

For a fixed $E/B$ we have that $B\hbox{-}\Alg[\mM_n(B)]$ is
 the category of Azumaya $B$-algebras 
          split over $E$ that are $n^2$-dimensional as $B$-vector spaces, and
\begin{equation}\eqlabel{4.22}
B\hbox{-}\Alg[\mM_n(B)] /\cong~{\rm is~bijective~to~}H^1(G,\PGL_n(E)),
\end{equation}
where $\PGL_n(E)=\GL_n(E)/\uU(E)$ is isomorphic to 
$\Aut_{E\hbox{-}\Alg}(E\ot_B \mM_n(B))$ by the Skolem-Noether Theorem according to
which every automorphism of $\mM_n(E)~(\cong E\ot_B \mM_n(B))$ is inner.
But again, we need to know how $G$ acts on $\PGL_n(E)$.\\

Since the elements of $\PGL_n(E)$ are equivalence classes of elements in $\GL_n(E)$, and they correspond to inner automorphisms of $\mM(B)$, we present $u$ in \equref{4.5} as such a class 
$[s]$, and then $u(t) = sts^{-1}$. After that \equref{4.5} gives:
\begin{eqnarray*}
 (gu)(t) &=& (u(tg))g^{-1} = (s(tg)s^{-1})g^{-1} = (sg^{-1})(tgg^{-1})(s^{-1}g^{-1})\\
 &=& (sg^{-1})t(s^{-1}g^{-1})
          = (sg^{-1})t(sg^{-1})^{-1} = v(t),
 \end{eqnarray*}
where $v$ is to $sg^{-1}$ as $u$ to $s$. Therefore the action of $G$ on $\PGL_n(E)$ can be defined by
\begin{eqnarray}
g[s]&=& [sg^{-1}]\eqlabel{4.23}\\
g[(s_{ij})]&=&[(s_{ij}g^{-1})],
\end{eqnarray}
since obviously $[sg^{-1}]=[(s_{ij}g^{-1})]$.

\refs

\bibitem [Borceux, Janelidze, 2001]{BJ} 
F. Borceux and G. Janelidze, Galois Theories, Cambridge Studies in Advanced Mathematics 72, Cambridge University Press, 2001.

\bibitem [Chase, Harrison, Rosenberg, 1965]{CHR}
S. U. Chase, D. K. Harrison, and A. Rosenberg, Galois theory and cohomology of commutative rings, {\sl Mem. Amer. Math. Soc.} {\bf 52} (1965), 15--33.

\bibitem [Chase, Sweedler, 1969]{CS}
S. U. Chase and M. E. Sweedler, Hopf algebras and Galois theory, 
{\sl Lecture Notes Math.} {\bf 97}, Springer, Berlin, 1969.

\bibitem [DeMeyer, Ingraham, 1971]{DI}
F. R. DeMeyer and E. Ingraham, Separable algebras over a commutative ring, 
{\sl Lecture Notes Math.} {\bf 181}, Springer, Berlin,  1971.

\bibitem [Greither, 1992]{Gre}
C. Greither, Cyclic Galois extensions of commutative rings, 
{\sl Lecture Notes Math.} {\bf 1534}, Springer, Berlin,  1992.

\bibitem [Grothendieck, 1971]{Gro}
A. Grothendieck, Rev\^ etements \'etales et groupe fondamental, SGA 1, expos\'e V, 
{\sl Lecture Notes Math.} {\bf 224}, Springer, Berln,  1971.

\bibitem [Jahnel, 2000]{Jah}
J. Jahnel, The Brauer-Severi variety associated with a central simple algebra, 
{\sl Linear Algebraic Groups and Related Structures} {\bf  52} (2000), 1--60.

\bibitem [Janelidze, Tholen, 2004]{JT}
G. Janelidze and W. Tholen, Facets of Descent III: Monadic descent for rings and algebras, 
{\sl Appl. Categ. Structures} {\bf 12} (2004), 461--467.

\bibitem [Knus, Ojanguren, 1974]{KO}
M.-A. Knus and M. Ojanguren, Th\'eorie de la descente et algbres dÕAzumaya, 
{\sl Lecture Notes Math.} {\bf 389}, Springer, Berlin, 1974

\bibitem [Magid. 1974]{M}
A. R. Magid, The separable Galois theory of commutative rings, Dekker, New York, 1974.

\bibitem [Serre, 1959]{S1}
J.-P. Serre, Groupes algŽbriques et corps de classes, Hermann, Paris, 1959.

\bibitem [Serre, 1964]{S2}
J.-P. Serre, Cohomologie Galoisienne, 
{\sl Lecture Notes Math.} {\bf 5}, Springer, Berln, 1964.

\bibitem [Weil, 1956]{W}
A. Weil, The field of definition of a variety, 
{\sl Amer. J. Math.} {\bf 78} (1956), 569--574.

\endrefs

\end{document}